\documentclass[10pt]{article}
\usepackage{amsmath, geometry, amssymb, fullpage, amsthm}
\usepackage{graphicx}

\def\bC{\mathbb{C}}

\def\bH{\mathbb{H}}

\def\bQ{\mathbb{Q}}

\def\bZ{\mathbb{Z}}

\def\cB{\mathcal{B}}

\def\cD{\mathcal{D}}
\def\cE{\mathcal{E}}

\def\cM{\mathcal{M}}

\def\cO{\mathcal{O}}

\def\cS{\mathcal{S}}

\def\fa{\mathfrak{a}}

\def\fd{\mathfrak{d}}

\def\ff{\mathfrak{f}}

\def\fp{\mathfrak{p}}

\def\GL{\operatorname{GL}}

\def\Gal{\operatorname{Gal}}

\def\Im{\operatorname{Im}}

\def\Nm{\operatorname{Nm}}

\def\SL{\operatorname{SL}}

\def\det{\operatorname{det}}

\def\dim{\operatorname{dim}}

\def\log{\operatorname{log}}

\newcommand{\beas}{\begin{eqnarray*}}
\newcommand{\eeas}{\end{eqnarray*}}

\newcommand{\ls}[2]{\left( \frac{#1}{#2} \right)}

\theoremstyle{plain}
\newtheorem{theorem}{Theorem}
\newtheorem{lemma}{Lemma}
\newtheorem{corollary}{Corollary}[theorem]

\theoremstyle{definition}
\newtheorem{definition}{Definition}

\title{Finding elementary formulas for theta functions associated to even sums of squares}
\author{Ila Varma \\
{\small ivarma@math.princeton.edu}\\
{\small Princeton University,}\\ 
{\small Department of Mathematics} \\
{\small Fine Hall, Washington Road} \\
{\small Princeton, NJ 08544} \\ 
{\small United States}
}

\begin{document}
\maketitle

\noindent \textbf{Abstract.} This article discusses the classical problem of how to calculate $r_n(m)$, the number of ways to represent an integer $m$ by a sum of $n$ squares from a computational efficiency viewpoint. Although this problem has been studied in great detail, there are very few formulas given for the purpose of computing $r_n(m)$ quickly. More precisely, for fixed $n$, we want a formula for $r_n(m)$ that computes in log-polynomial time (with respect to $m$) when the prime factorization of $m$ is given. Restricting to even $n$, we can view $\theta_n(q)$, the theta function associated to sums of $n$ squares, as a modular form of weight $n/2$ on $\Gamma_1(4)$. In particular, we show that for only a small finite list of $n$ can $\theta_n$ be written as a linear combination consisting entirely of Eisenstein series and cusp forms with complex multiplication. These are the only $n$ that give rise to ``elementary'' formulas for $r_n(m)$, i.e. formulas such that for a prime $p$, $r_n(p)$ can be calculated in $\cO(\log(p))$-time. Viewing $\theta_n(q)$ as one of the simpler examples of modular forms that are not strictly Eisenstein, this result motivates the necessity of a log-polynomial time algorithm that directly calculates the Fourier coefficients of modular forms in the generic situation when there is no such formula, as described in Couveignes and Edixhoven's forthcoming book (for level 1 cases) and Peter Bruin's Ph.D. thesis (for higher level, including 4). 

\section*{Introduction}

We are interested in the even case of the classical ``sums of squares'' problem. For positive $n \in \bZ$ and positive $m \in \bZ$, let
\begin{equation} \label{def}
	r_n(m) := \#\{{\bf x} = (x_1,...x_n) \in \bZ^n : x_1^2 + x_2^2 + ... + x_n^2 = m\}. \end{equation}
We wish to understand $r_n(m)$ by analyzing the generating function attached to $r_n(m)$, denoted by $\theta_n$. As a complex function on the open unit disk $\cD \subseteq \bC$, $\theta_n$ is the theta function attached to the standard $n$-dimensional quadratic form,
\begin{equation} \label{theta}
	\theta_n(q) = \sum_{m = 0}^\infty r_n(m)q^m = 1 + 2n\cdot q + 4{n\choose 2} \cdot q^2 + 8 {n\choose 3} \cdot q^3 + \left[2^4 {n \choose 4} + 2n \right] \cdot q^4 + \hdots .\end{equation}
Equivalently, one can define $\theta_n(q)$ by using the multiplicative property of $r_n(m)$:
	$$\theta_n(q) = \theta_1(q)^n = \left(1 + 2q + 2q^4 + 2q^9 + ...\right)^n.$$ 
Furthermore, if we view $\cD$ as the image of the upper half plane $\bH = \{z \in \bC : \Im(z) > 0 \}$ under the map $z \mapsto e^{2\pi i z}$, then we may write,
	$$\theta_n(z) = \sum_{m = 0}^\infty r_n(m)e^{2\pi i m z} .$$ 
It satisfies the equations (see \cite{miyake}, 3.2),
\begin{equation} \label{theta eqns}
	\theta_n(-1/4z) = (2z/i)^{n/2}\theta_n(z), \quad \quad \quad \theta_n(z + 1) = \theta_n(z).\end{equation} 
Since we restrict ourselves to the case where $n$ is even, there is no need to choose a square root. The above symmetric properties illustrate that $\theta_n$ (as a function of the $\bH$-coordinate $z$) is a modular form of weight $k = \frac{n}{2}$ on the congruence subgroup $\Gamma_1(4)$, consisting of matrices $\gamma \in \SL_2(\bZ)$ such that $\gamma \equiv \left(\begin{smallmatrix} 1 & \ast \\ 0 & 1 \end{smallmatrix} \right) \pmod{4}$ (see \cite{ds}, 1.2 or \cite{miyake}, 3.2).

We wish to analyze for which $n$ does $\theta_n(z)$ have coefficients that are ``straightforward to compute'', i.e. for fixed $n$, when $r_n(m)$ can be written as a function on $m \in \bZ$ such that each Fourier coefficient $a_m = r_n(m)$ of $\theta_n(z)$ can be computed in $\cO(\log(m))$-time, given the prime factorization of $m$. Equivalently, for a prime $p$, we want formulas for $r_n(p)$ that can be computed in time polynomial to the number of digits of $p$. We therefore establish the following definition.

\begin{definition} A modular form $f$ of integral weight is \emph{elementary} if and only if $f$ is a linear combination of Eisenstein series and cusp forms with complex multiplication (CM cusp forms) as defined in Section \ref{cm}. \end{definition} 

Denote the space of modular forms on a congruence subgroup $\Gamma$ of weight $k$ as $\cM_k(\Gamma),$ and let its subspace of cusp forms be $\cS_k(\Gamma)$. The orthogonal complement (with respect to the Petersson inner product) is denoted $\cE_k(\Gamma)$, the space of Eisenstein series. Finally, we define the subspace $\cS_k^{cm}(\Gamma) \subset \cS_k(\Gamma)$ as the CM subspace, i.e. the space generated by those cusp forms which are invariant under twisting by a quadratic character. 

The definition of elementary comes up quite naturally in a variety of places. In particular, elementary modular forms are in correspondence with 2-dimensional potentially abelian $\ell$-adic representations of $\Gal(\overline{\bQ}/\bQ)$, i.e. after passing to an open subgroup (equivalently, to a finite extension of $\bQ$), the representation factors through an abelian quotient. These turn out to be particularly nice examples of geometric representations (see \cite{fontainemazur}, \S6 for details).

Furthermore, in \cite{lacunarity}, Serre calls cusp forms \emph{lacunary} if the arithmetic density
of the nonzero coefficents in the $q$-expansion is zero, and proves that cusp form $f \in \cS_k(\Gamma)$ is lacunary if and only if $f \in \cS_k^{cm}(\Gamma)$. While it is false that $\theta_n$ is lacunary for any $n \geq 4$ (due to Lagrange), it is true that $\theta_n$ is elementary if and only if the cuspidal part in its decomposition are lacunary, i.e. contribute to the value of $r_n(m)$ very rarely from a density viewpoint, making most coefficients of $\theta_n$ a sum of Eisenstein coefficients. The following theorem is our main result.

\begin{theorem} \label{main} 
Suppose $n$ is even. Then $\theta_n$ is elementary if and only if $n = 2,4,6,8,$ or $10$. \end{theorem} 

For the listed exceptional cases, the following formulas for the Fourier coefficients $r_n(m)$ were produced originally by Jacobi (for $n = 2, 4, 6, 8$) and Liouville ($n = 10$) in the 19th century (see \cite{jacobi} and \cite{liouville1864}). Let $\chi_4$ denote the nontrivial Dirichlet character of conductor 4 extended multiplicatively to all positive integers.

\begin{eqnarray} 
	r_2(m) &=& 4 \left(\sum_{d \mid m}\chi_4(d)\right). \\
	r_4(m) &=&   8 \left(\sum_{\stackrel{d \text{ odd}}{d \mid m}} d\right) + 16 \left(\sum_{\stackrel{d \text{ odd}}{d \mid \frac{m}{2}}} d \right). \\
	r_6(m) &=& 16\left(\sum_{d \mid m} \chi_4\left(\frac{m}{d}\right)d^2 \right) - 4\left(\sum_{d \mid m} \chi_4(d)d^2 \right). \\
	r_8(m) &=& 16 \left( \sum_{d \mid m} d^3\right) - 32 \left( \sum_{d \mid \frac{m}{2}} d^3 \right) + 256 \left( \sum_{d \mid \frac{m}{4}} d^3 \right). \\
	r_{10}(m) &=& \frac{4}{5} \left(\sum_{d \mid m} \chi_4(d)d^4 \right) + \frac{64}{5}\left(\sum_{d \mid m} \chi_4\left(\frac{m}{d}\right) d^4 \right) + \frac{8}{5} \left(\sum_{\stackrel{\fd \in \bZ[i]}{\Nm(\fd) = m}} \fd^4 \right).  \end{eqnarray}
	
	It has been known that for even $n < 10$, the theta function is purely Eisenstein, and therefore elementary (see \cite{rankin}). For even $n > 8$, Liouville, Glaisher, and Ramanujan tried to give formulas in the same vein as those given for $r_2, r_4, r_6,$ and $r_8$; however, they began to involve various arithmetic and analytic functions for the cuspidal coefficients which were exceedingly difficult to compute, even for small $m$ (see \cite{glaisher} and \cite{ramanujan}). General formulas for all even $n$ were also remarkably difficult due to the contribution from the cusp form, though the study of the Eisenstein part of $\theta_n(q)$ allowed for understanding of the general growth of these functions (see \cite{milnes}). However, as we will show, CM cusp forms have coefficients given by a Hecke character on an associated imaginary quadratic field, hence it is still possible to have log-polynomial computable coefficients when the cuspidal part of $\theta_n(z)$ has CM, i.e. elementary formulas can exist for $r_n(m)$ when $n > 8$. 
To prove the main theorem, we compute the dimensions of $\cE_{n/2}(\Gamma_1(4)) \oplus \cS^{cm}_{n/2}(\Gamma_1(4))$. It turns out that the first cusp form on $\Gamma_1(4)$ is of weight 5 and is invariant under twisting by the quadratic character attached to $\bQ(i)$; in fact, $\cS^{cm}_5(\Gamma_1(4)) = \cS_5(\Gamma_1(4))$. Thus, we can already conclude that $\theta_2, \theta_4, \theta_6, \theta_8,$ and $\theta_{10}$ are elementary since they are modular forms on $\Gamma_1(4)$ of weight less than or equal to 5.  

For the converse, we first show that $\cS_{n/2}^{cm}(\Gamma_1(4))$ is one-dimensional when $\frac{n}{2} \equiv 1 \bmod{4}$ and trivial otherwise. Since the Riemann-Roch formula implies that $\cE_{n/2}(\Gamma_1(4))$ has dimension at most 3, we only need to show that the modular form $\theta_n$ is not in the span of a basis of (at most 3) Eisenstein series and 1 CM eigenform (if $\frac{n}{2} \equiv 1\bmod{4}$). This can be done by constructing bases for $\cE_{n/2}(\Gamma_1(4)) \oplus \cS^{cm}_{n/2}(\Gamma_1(4))$ and comparing coefficients. Thus, $\theta_n$ is only elementary in the exceptional cases where the entire cuspidal subspace of $\cM_5(\Gamma_1(4))$ is CM. Note that for $n > 10$, $\cS^{cm}_{n/2}(\Gamma_1(4)) \subsetneq \cS_{n/2}(\Gamma_1(4))$.

We can also prove the backwards direction of Theorem \ref{main} by looking at the action of certain Hecke operators on the CM cuspidal subspace. In particular, we can show that the Hecke action on the cuspidal part of $\theta_n(q)$ is not the same as it would be on an element of $\cS^{cm}_{n/2}(\Gamma_1(4))$ when $n > 10$.

Section 1 describes cusp forms with complex multiplication and pertinent properties, including their relation to Hecke characters of imaginary quadratic fields. In Section 2, we show that $\theta_n(q)$ is elementary for even $n \leq 10$ by a dimension argument. In Section 3, we show that $\theta_n(q)$ for $n > 10$ cannot be elementary by comparing it with the Eisenstein and CM cuspidal subspaces of $\cM_{n/2}(\Gamma_1(4))$ concluding the first proof of Theorem \ref{main}. In Section 4, we produce another proof of the fact that $\theta_n(q)$ is not elementary for all $n >10$ and even by contrasting the action of the Hecke operator $T_3$ on the CM cuspidal space with its action on $\theta_n$. Finally, in Section 5, we discuss the context of this result and the notion of elementary in modern computations of Fourier coefficients of modular forms. 

The author is grateful to the mathematics department at Leiden University for their hospitality during this research. In particular, the author thanks Prof. Bas Edixhoven for the proposal of this problem as well as many valuable discussions. She also thanks Lenny Taelman and Peter Bruin, who were also extremely helpful in providing many suggestions and advice, as well as the reviewers for their comments. This research was completed using funding from the U.S. Fulbright Grant and the HSP Huygens Fellowship. 

\section{Cusp forms with complex multiplication} \label{cm}

We recall the theory of CM cusp forms, and describe their relation to Hecke characters of imaginary quadratic fields. For a more detailed explanation, see \cite{ribet}. 

Fix positive integers $k$ and $N$, and let $\varepsilon$ be a Dirichlet character on $(\bZ/N\bZ)^\times$ such that $\varepsilon(-1) = (-1)^k$. Recall that $\cS_k(\Gamma_0(N), \varepsilon)$ is the $\varepsilon$-eigenspace of the diamond operator $\langle d \rangle$ where $d \in (\bZ/N\bZ)^\times$ (recall that $\Gamma_0(N)/\Gamma_1(N) \cong (\bZ/N\bZ)^\times$ and the diamond operators represent the action of $\Gamma_0(N)$ on $\cS_k(\Gamma_1(N))$). For a normalized eigenform $f \in \cS_k(\Gamma_0(N),\varepsilon)$, let $K_f$ be the number field generated by its Hecke eigenvalues. Note that $K_f$ is either totally real or has complex multiplication, and it contains the image of the Nebentypus $\varepsilon$. It is real if and only if $\varepsilon$ factors through $\{\pm 1\} \subseteq \bC^\times$ and
	$$\varepsilon(p) a_p = a_p \quad \forall \mbox{ primes } p \nmid N,$$
where $a_p$ denotes the coefficient of $q^p$ in the Fourier expansion of $f$ (see Prop 3.2, \cite{ribet}).

For a newform $f \in \cS_k(\Gamma_0(N), \varepsilon)$, we can twist by a Dirichlet character $\varphi$ defined on $(\bZ/D\bZ)^\times$ as follows: 
	$$f \otimes \varphi = \sum_{n=1}^\infty \varphi(n) a_nq^n \in \cS_k(\Gamma_0(ND^2),\varepsilon\varphi^2).$$
Moreover, $f \otimes \varphi$ is an eigenform since the action of the Hecke operator $T_p$ for $p \nmid ND$ has eigenvalue $\varphi(p)a_p$. 

An eigenform $f$ has \emph{complex multiplication} (or CM) by $\varphi$ if $f \otimes \varphi = f$. One must check that $\varphi(p)a_p = a_p$ (equivalently, either $\varphi(p) = 1$ or $a_p = 0$) for a set of primes of density 1 in order to conclude that $f(z) = \sum_{n \geq1} a_nq^n$ has CM by $\varphi$. Furthermore, this implies that $\varepsilon \varphi^2 = \varepsilon$, so $\varphi$ must be a quadratic character.

\noindent \textbf{Remark.} Using $\Gamma_1(N)$, we can define the notion of a CM cusp form on all $\Gamma$, namely in the direct limit over all levels. 

The theory of $\ell$-adic representations gives a way to generate cusp forms on $\Gamma_1(N)$. In particular, Hecke characters of imaginary quadratic fields give rise to CM cusp forms.

\subsection{Cusp forms attached to Hecke characters}

To every algebraic Hecke character $\psi$, we can attach a cusp form in a canonical way by describing the $q$-series expansion in terms of the values the character takes outside of primes dividing the conductor. More precisely, if $\psi$ has conductor $\ff$ and $\infty$-type $t$ on an imaginary quadratic field $K$, define the $q$-series 
\begin{equation} \label{CM form}
	f_{K,\psi}(z) = \sum_{\tiny \begin{matrix} \fa \text{ integral} \\ \text{coprime to $\ff$}\end{matrix}} \psi(\fa)\cdot q^{\Nm(\fa)} \quad (q = e^{2\pi i z}, \Im(z) > 0). \end{equation}
This is well-defined as a function on the $z$-coordinate from the upper half plane to $\bC$. One can further check that it has certain symmetric properties that allow $f_{K,\psi}$ to be viewed as a modular form. 

Let $K$ and $\psi$ be as above, let $\omega_\psi$ be the Dirichlet character of conductor $\Nm(\ff)$ attached to $\psi$, defined by
	 $$\omega_\psi(a) = \psi((a))/a^t \quad \forall a \in \bZ.$$ Note that $t \in \bZ_{>0}$ due to algebraicity. Furthermore, let $\chi_K$ be the quadratic character attached to $K$, defined by the Kronecker symbol, $$\chi_K(p) = \ls{d_K}{p},$$ for odd primes $p$ not dividing the discriminant $d_K$, and extended to $\bZ$ in the usual way. 
\begin{theorem}[Hecke \cite{hecke}, Shimura \cite{jacobianshimura} \& \cite{realshimura}] \label{hs} The $q$-series $f_{K,\psi}(z)$ is a newform of weight $t + 1$ and level $d\cdot \Nm(\ff)$ in the space 
	$$f_{K,\psi} \in \cS_{t+1}(\Gamma_0(d\cdot \Nm(\ff)), \chi_K\cdot \omega_\psi)) \subset \cS_{t+1}(\Gamma_1(d \cdot \Nm(\ff))).$$ 
In addition, distinct cusp forms $f_{K,\psi}$ arise from distinct pairs $(K, \psi)$. 
\end{theorem}
Thus, given a weight $k$, level $N$, and Nebentypus $\varepsilon$, we can construct a basis for the space of CM cusp forms using the criterion in the above theorem.
\begin{corollary}\label{CM basis} For any positive integer $r$, 
	$$f_{K,\psi}(r\cdot z) = \sum_{\fa} \psi(\fa)\cdot q^{r \cdot\Nm(\fa)} \in \cS_{t+1}(\Gamma_0(N),\varepsilon) \quad \Longleftrightarrow \quad r\cdot d \cdot \Nm(\ff) \mid N \mbox{  and  } \chi_K \cdot \omega_\psi = \varepsilon$$ 
if $\psi$ has $\infty$-type $t$. The second equality takes place while viewing $\chi_K$, $\omega_\psi$ and $\varepsilon$ as characters on $\bZ$. Thus, for all primes $p \nmid N$, $\chi_K(p) \cdot \omega_\psi(p) = \chi(p)$.  \end{corollary} 

Note that the cusp forms $f_{K,\psi}(r\cdot z)$ have CM by $\chi_K$. (By construction, the coefficient of $q^p$ in the $q$-series expansion of $f_{K,\psi}(r\cdot z)$ is 0 if no ideal of $K$ has norm equal to $p$. Since $\chi_K(p) = -1$ exactly when this holds for $p$, $a_p = \chi_K(p) a_p$.) However, it is not at all obvious that these are all the cusp forms with complex multiplication in $\cS_k(\Gamma_1(N))$. Using the theory of Galois representations, Ribet proved that a newform $f$ has CM by an imaginary quadratic field $K$ if and only if it arises from a Hecke character on $K$ (see 4.4 and 4.5 in \cite{ribet}). 

\section{The exceptional cases of $n < 10$} \label{exceptional}

One can try to show from the classical formulas for $r_n(m)$ that $\theta_n(q)$ are elementary when $n < 12$ and even. Instead, simply considering the dimension of the first few weighted spaces of modular forms on $\Gamma_1(4)$ allows us to make the same conclusion that $n = 2, 4, 6, 8,$ and 10 are elementary. From the Riemann-Roch formula specialized to $\Gamma_1(4)$, we have the following dimension formulas (see \cite{ds}, \cite{miyake}, or \cite{di} for proof): 

\begin{lemma} \label{dim} 
The dimensions of $\cM_k(\Gamma_1(4))$ and its subspace of cusp forms for each $k \in \bZ_{> 0}$ are as follows: 
$$\begin{array}{c c}
\dim_{\bC}(\cM_k(\Gamma_1(4))) = 
	\begin{cases} \frac{k+2}{2} & \mbox{ if $k$ is even }\\
                               \frac{k+1}{2} & \mbox{ if $k$ is odd,}
        \end{cases} & 
\dim_{\bC}(\cS_k(\Gamma_1(4))) = 
	\begin{cases} 0 & \mbox{ if $k \leq 4$} \\
                               \frac{k-4}{2} & \mbox{ if $k\geq 3$ is even }\\
                               \frac{k-3}{2} & \mbox{ if $k \geq 3$ is odd.}\\ \end{cases}\end{array}$$ \end{lemma}
                               
In particular, this implies that $\cM_k(\Gamma_1(4))$ is entirely Eisenstein for $k < 5$, and $\cS_5(\Gamma_1(4))$ has dimension 1. Thus, the backwards direction of Theorem \ref{main} follows from proving that the generator of $\cS_5(\Gamma_1(4))$ is a CM cusp form. By Corollary \ref{CM basis}, the CM newform attached to the Hecke character of trivial conductor and $\infty$-type 4 on $\bQ(i)$ 
	\begin{eqnarray*} f_{\bQ(i),\psi}(z) &=& \sum_{\fa \text{ integral}} \psi(\fa)\cdot q^{\Nm(\fa)} \quad \quad \mbox{where} \quad \psi((\alpha)) = \alpha^{4}, \\
					              &=& \frac{1}{4} \sum_{m \geq 1}\left(\sum_{\stackrel{\fd \in \bZ[i]}{\Nm(\fd) = m}} \fd^{4} \right) e^{2\pi i mz}. \end{eqnarray*}
It is an element of $\cS_{5}(\Gamma_0(4),\chi_4) = \cS_5(\Gamma_1(4))$, where $\chi_4$ denotes the nontrivial Dirichlet character of conductor 4. Thus, since $f_{\bQ(i),\psi}$ is a newform, it is the generator of the entire cuspidal space, hence $\cS_5(\Gamma_1(4)) = \cS^{cm}_5(\Gamma_1(4))$. 

From the above argument and the fact that $\theta_n(q) \in \cM_{n/2}(\Gamma_1(4))$, we can therefore conclude that $\theta_2(q)$, $\theta_4(q)$, $\theta_6(q)$, $\theta_8(q)$, and $\theta_{10}(q)$ are elementary.

\section{$\theta_n(q)$ is not elementary for $n > 10$}

The previous section illustrates that $\theta_n(q)$ is elementary if $n = 2, 4, 6, 8$, and 10 using the dimension of $\cM_k(\Gamma_1(4))$ for $\frac{n}{2} = k \leq 5$. However, Lemma \ref{dim} also implies that $\cE_k(\Gamma_1(4))$ has dimension 3 for even $k > 2$ and dimension 2 for odd $k > 2$. We can go further and prove the entire theorem using a dimension argument via the following lemma. 

\begin{lemma} \label{hecke} 
The dimension of $\cS^{cm}_k(\Gamma_1(4))$ is $1$ if $k \equiv 1 \bmod 4$ for $k \geq 5$ and 0 otherwise. \end{lemma}

\begin{proof} Assume $N = 4$. The two Dirichlet characters on $(\bZ/4\bZ)^\times$ are the trivial character ${\bf 1}$ and the primitive character of conductor 4, $\chi_4$. Recall that the decomposition of $\cS_k(\Gamma_1(N)) = \bigoplus_{\varepsilon} \cS_k(\Gamma_0(N),\varepsilon)$
occurs over all characters defined modulo $N$ such that $\varepsilon(-1) = (-1)^k$, thus
	$$\cS_k(\Gamma_1(4)) = \begin{cases} \cS_k(\Gamma_0(4), {\bf 1}) & \mbox{ if $k$ is even} \\
								  \cS_k(\Gamma_0(4), \chi_4) & \mbox{ if $k$ is odd.} \end{cases}$$
Let $\varepsilon$ denote the correct Dirichlet character such that $\cS_k(\Gamma_1(4)) = \cS_k(\Gamma_0(4), \varepsilon)$ depending on the parity of $k$. We first want to understand the CM subspace $\cS^{cm}_k(\Gamma_0(4),\varepsilon)$ by considering the newforms which generate it according to Corollary \ref{CM basis}.

For a newform $f_{K,\psi}(r \cdot z) \in \cS^{cm}_k(\Gamma_0(4),\varepsilon)$, the imaginary quadratic field $K$ must have discriminant $d_K$ such that $d_K \mid 4$. Thus, $K = \bQ(i)$ since there is no field of discriminant $-2$. Note that $\chi_K$ is a nontrivial Dirichlet character of conductor 4, so $\chi_K = \chi_4$.

Since the discriminant of $K$ is $-4$, we are left with finding all Hecke characters $\psi$ on $K$ of conductor $1$ such that $\omega_\psi = {\bf 1}$ if the $\infty$-type $t$ of $\psi$ is even, and $\omega_\psi = \chi_4$ if $t$ is odd. By definition, if $\alpha \in K^\times$ then since $(\alpha) = (u \alpha)$ for all $u \in \cO_K^\times$,
	$$\psi((\alpha)) = \alpha^t = (u \alpha)^t  = \psi((u \alpha)).$$ 
Since $\cO_K^\times = \langle i \rangle$, we can conclude that $4 \mid t$, since if $u = \pm i$, $i^t  = 1$ if and only if $t \equiv 0 \bmod{4}$. In particular, there are no Hecke characters of odd $\infty$-type, hence $\cS_k^{cm}(\Gamma_0(4),\varepsilon)$ is trivial for even $k$. In addition, $\cS_k^{cm}(\Gamma_0(4),\varepsilon)$ can only be non-trivial when $k \equiv 1 \pmod{4}$ and $k > 1$. 

We now prove that there is a unique Hecke character with the properties described above for each $\infty$-type $t \in \bZ$ such that $4 \mid t$. Let $t$ be fixed such that $t \equiv 0 \pmod{4}$. Then $\chi_K \cdot \omega_\psi = \chi_4$ for any $\psi$, hence the attached character $\omega_\psi$ must be trivial. Thus, since $\cO_K = \bZ[i]$ is a PID, the algebraic Hecke characters on $K$ are exactly those defined as 
	$$\psi_t: (\alpha) \longmapsto \alpha^t \quad \forall \alpha \in K^\times \quad \quad \mbox{where $t \in \bZ$ and $4 \mid t$.}$$
Using Corollary \ref{CM basis}, we conclude that the dimension of $S_k^{cm}(\Gamma_0(4),\varepsilon) \subseteq S_k(\Gamma_1(4))$ is 1 if $k \equiv 1 \bmod 4$ and $k \geq 5$ and is $0$ otherwise. \end{proof}

\subsection{Proof of Theorem \ref{main} via a dimension argument} 

We want to show that one must calculate the coefficients of a non-CM cusp form in order to calculate the representation number for each integer by the quadratic form $x_1^2 + ... + x_n^2$ when $n > 10$ and even. To do this, we first construct a canonical basis for the space $\cE_{n/2}(\Gamma_1(4)) \oplus \cS_{n/2}^{cm}(\Gamma_1(4))$. 

Consider first if the weight $\frac{n}{2} = k > 2$ is even. Then $\cM_k(\Gamma_1(4)) = \cE_k(\Gamma_0(4), {\bf 1}_4)$. By Lemma \ref{dim}, the dimension of the space is 3.

Let $\ell$ be a prime. For $f \in \cE_k(\Gamma_1(4))$, the associated 2-dimensional reducible $\ell$-adic representation will have the form $\rho_\ell = \varepsilon_1 \oplus \varepsilon_{2} \chi_\ell^{k-1}$, where the Dirichlet characters $\varepsilon_1 \cdot \varepsilon_2 = {\bf 1}_4$ (see \cite{ribet}, Pp. 28), and $\chi_\ell$ denotes the $\ell$-adic cyclotomic character. Furthermore, since the product of the conductors of $\varepsilon_1$ and $\varepsilon_2$ must divide the desired level $N = 4$, both characters must be trivial. Consider the representation
	$$\rho = 1 \oplus \chi_\ell^{k-1}: \Gal(\overline{\bQ}/\bQ) \longrightarrow \GL_2(\bQ_\ell).$$
The $L$-series attached to this represenation is 
	$$ L(s, 1\oplus \chi_\ell^{k-1}) = \prod_p \frac{1}{1-p^{-s}} \cdot \prod_p \frac{1}{1-p^{k-1}p^{-s}} = \sum_{m \geq 1} \left(\sum_{d \mid m} d^{k-1}\right) \frac{1}{m^s}.$$
Taking the inverse Mellin transform gives us the $q$-series
\begin{equation}\label{evenweight}
E(z) = a_0 + \sum_{m \geq 1} \left(\sum_{d\mid m} d^{k-1}\right)q^m. \end{equation}
In order to calculate $a_0$, we consider the functional equation for $L$-series and check that its Lambda function $\Lambda(s,\rho)$ at $s = 0$ has residue $-a_0 = -\zeta(1-k)/2$, in terms of the Riemann zeta function. Thus, $a_0 = -\frac{\bf b_k}{2k}$ where ${\bf b_k}$ is the $k$th Bernoulli number (see Section \ref{proof2}).

Using the classical converse theorem of Weil, we conclude that
$E(z)$ lies inside $\cE_{k}(\SL_2(\bZ)) \subseteq \cE_k(\Gamma_1(4))$ since the conductor of the trivial character is 1 (see \cite{miyake}, \S 4.7, and \cite{milneMF}, Ch. 9). However, this implies that the $q$-series $E(q^2)$ and $E(q^4)$ are also in $\cE_k(\Gamma_1(4))$, and furthermore, they are linearly independent. These are all eigenforms, and since the dimension of $\cE_k(\Gamma_1(4))$ is three, $\cB = \{E(q), E(q^2), E(q^4)\}$ is a basis.
	
Since $\theta_n(q)$ is a modular form of weight $k = n/2$, we are considering the case when $n \equiv 0 \bmod{4}$ for $n >8$. Thus, if $\theta_n(q)$ is elementary, then there exist constants $a, b, c \in \bC$ such that $$\theta_{n}(q) = a \cdot E(q) + b \cdot E(q^2) + c \cdot E(q^4).$$ Consider the determinant of the matrix of coefficients of the 4 $q$-series:
	$$\det(M) = \begin{array}{| c c c c | cl} 
 	2n &  4{n \choose 2} & 8{n \choose 3} & 16 {n \choose 4} + 2n & & (\text{\small \em coefficients of  $\theta_n$})\\ 
	1 & 1+2^{n/2 - 1} & \displaystyle1 + 3^{n/2- 1} &1 + 2^{n/2- 1} + 4^{n/2 - 1} & & (\text{\small \em coefficients of  $E(q)$})\\
	0 &  1 & 0 & 1 + 2^{n/2 - 1} & & (\text{\small \em coefficients of  $E(q^2)$})\\
	0 & 0 & 0 & 1 & & (\text{\small \em coefficients of $E(q^4)$})\end{array}$$
The determinant of $M$ is zero if and only if there is a linear dependence amongst the coefficients, i.e. if the coefficients of $q$, $q^2$, $q^3$, and $q^4$ of $\theta_n$ can be written in terms of coefficients of the forms in $\cB$. Solving for the determinant gives 
	$$\det(M) = - \left(2n\cdot\left(1 + 3^{\frac{n}{2} - 1}\right) - 8{n \choose 3}\right) = - 2n  - 2n\cdot 3^{\frac{n}{2} - 1} + 8 \frac{n(n-1)(n-2)}{6}.$$
Note that the growth of this function is dominated by $-2n \cdot 3^{\frac{n}{2} - 1}$, and for $n > 8$, $\det(M)$ as a function on $n$ is monotonically decreasing. The determinant is 0 when $n = 4$ and $8$, and negative for $n = 12$, thus $\det(M)$ is nonzero for all larger $n$. This implies that $\theta_n \notin \cE_{n/2}(\Gamma_1(4))$ for $n \equiv 0 \pmod{4}$ when $n > 8$, hence by Lemma \ref{hecke}, this implies that $\theta_n$ is not elementary for $n > 8$ when $n \equiv 0 \pmod{4}$ (as there are no CM cuspidal forms of even weight). \\

Now assume $n \equiv 2 \pmod{4}$, i.e the weight of $\theta_n(q)$ is odd. When $n \equiv 6 \pmod{8}$, there is no CM subspace of $\cS_{n/2}(\Gamma_1(4))$ since $k = \frac{n}{2} \equiv 3 \bmod{4}$. We first check that $\theta_n \notin \cE_{n/2}(\Gamma_1(4))$ for all $n \equiv 2 \bmod{4}$ and $n > 6$, which will reduce the problem to whether there is a contribution by a CM cusp form when $n \equiv 2 \bmod{8}$.

When $n \equiv 2 \pmod{4}$, the $\cE_{k}(\Gamma_1(4))$ is a two-dimensional subspace by Lemma \ref{dim}. Since the weight $k$ is odd, $\cE_k(\Gamma_1(4)) = \cE_k(\Gamma_0(4),\chi_4)$. Thus, we want to construct 2 linearly independent Eisenstein eigenforms of level 4 with Nebentypus $\chi_4$.

Let $\ell$ be a prime. Similar to before, a 2-dimensional reducible $\ell$-adic representation associated to a basis element of $\cE_k(\Gamma_0(4),\chi_4)$ will have the form $\varepsilon_1 \oplus \varepsilon_2 \chi_\ell^{k-1}$, where the characters $\varepsilon_1 \cdot \varepsilon_2 = \chi_4$ and the product of their conductors divide the level (see \cite{ribet}, Pp. 28). Thus, we have the following two possible representations:
\begin{eqnarray*} \rho_1 := \chi_4 \oplus \chi_\ell^{k-1}: \Gal(\overline{\bQ}/\bQ) & \longrightarrow& \GL_2(\bQ_\ell), \\
                            \rho_2 := 1 \oplus \chi_4\chi_\ell^{k-1}: \Gal(\overline{\bQ}/\bQ) & \longrightarrow & \GL_2(\bQ_\ell). 
\end{eqnarray*} 
The $L$-series attached to these representations are 
	\begin{eqnarray*} L(s,\rho_1) & = & \displaystyle\prod_p \frac{1}{1-\chi_4(p)p^{-s}} \cdot \displaystyle\prod_p \frac{1}{1 - p^{k-1}p^{-s}}  =  \displaystyle\sum_{m \geq 1} \left(\sum_{d \mid m} \chi_4\left(\frac{m}{d}\right)d^{k-1}\right) \frac{1}{m^s}, \\	
	   L(s,\rho_2) & = & \displaystyle\prod_p \frac{1}{1 - p^{-s}} \cdot \displaystyle\prod_p \frac{1}{1 - \chi_4(p)p^{k-1}p^{-s}}  =  \displaystyle\sum_{m \geq 1} \left(\sum_{d \mid m} \chi_4(d) d^{k-1} \right)\frac{1}{m^s}. \end{eqnarray*}
Applying inverse Mellin transforms to the above $L$-functions produces two $q$-series
\begin{eqnarray}\label{weight3}
 	E_1(q) &=& a_{0,1} + \sum_{m \geq 1} \left(\sum_{d \mid m} \chi_4\left(\frac{m}{d}\right)d^{k-1}\right)q^m, \\
	E_2(q) &=& a_{0,2} + \sum_{m \geq 1} \left(\sum_{d \mid m} \chi_4(d)d^{k-1}\right)q^m. \label{weight32} \end{eqnarray} 
One can calculate $a_{0,1}$ and $a_{0,2}$ by considering the functional equations for the $L$-series and calculating the residues of the associated Lambda functions $\Lambda(s,\rho_1)$  and $\Lambda(s,\rho_2)$ at $s = 0$. One can check that $-a_{0,1} = -L(-2,\chi_4)/2$, in terms of the $L$-series associated to the Dirichlet character $\chi_4$, and $-a_{0,2} = 0$. Thus, $$a_{0,1} = -\frac{\bf b_k^{\chi_4}}{2k} 
\quad \text{and} \quad a_{0,2} = 0, \quad \quad \text{where} \quad \frac{t \cdot e^{t}}{e^{4t} - 1} - \frac{t \cdot e^{3t}}{e^{4t} - 1} = \sum_{k = 0}^\infty {\bf b_k^{\chi_4}} \frac{t^k}{k!},$$ is the $k$th Bernoulli number associated to the nontrivial Dirichlet character of conductor 4 defined by the above equation (see \cite{stein}, Ch. 5). 

The converse theorem of Weil implies that both $E_1(q)$ and $E_2(q)$ are distinct eigenforms on $\Gamma_1(4)$ of weight $k$, hence the two series are linearly independent (see \cite{miyake}, \S 4.7, and \cite{milneMF}, Ch. 9). By construction, $E_1(q)$ and $E_2(q) \in \cE_k(\Gamma_1(4))$, thus we take $\cB = \{E_1(q), E_2(q)\}$ as a basis for the Eisenstein space, $\cE_{n/2}(\Gamma_1(4))$. 

To show $\theta_n \notin \cE_{n/2}(\Gamma_1(4))$, we check if the following matrix of coefficients for $q$, $q^2$, and $q^3$ has non-zero determinant:
	$$\det(M') = \begin{array}{| c c c  | cl} 
	 2n & 4{n \choose 2} & 8{n \choose 3}  & & (\mbox{\small \em coefficients of  $\theta_n$})\\
	1 & 2^{n/2 - 1} & -1 + 3^{n/2 - 1}  & & (\mbox{\small \em coefficients of  $E_1(q)$})\\
	1 &  1  & 1 - 3^{n/2 - 1} & & (\mbox{\small \em coefficients of  $E_2(q)$}) \end{array} $$
Solving for the determinant,
	\begin{eqnarray*} \det(M') &=& 2n \left(2^{n/2 - 1} + 1 \right)\left(1-3^{n/2 - 1}\right)  - 8{n \choose 2}\left(1 - 	3^{n/2 - 1}\right) + 8{n \choose 3}\left(-2^{n/2 - 1}\right) \\ 
	&=& \left(-4n^2 + 6n\right) + 2^{n/2 - 1}\cdot\left(-\frac{4}{3}n^3 + 4n^2 - \frac{2}{3}n\right) + 3^{n/2 - 1}\cdot\left(4n^2 - 6n\right) + 6^{n/2 - 1} \cdot \left(-2n\right). \end{eqnarray*}
Note that for $n > 6$, $\det(M')$ as a function on $n$ is governed by the growth of $-2n\cdot 6^{n/2 - 1}$. One can check that for $n = 6$, $\det(M') = 0$, and for $n > 6$, $\det(M') < 0$ and is monotonically decreasing. We conclude that for $n \equiv 2 \bmod{4}$ and $n > 6$, $\theta_n$ is not a linear combination of Eisenstein series. By Lemma \ref{hecke}, this implies that for all positive $n \equiv 6 \bmod{8}$ except $n = 6$, $\theta_n$ is not elementary. \\
	
	It remains to show that $\theta_n$ is not elementary when $n \equiv 2 \bmod{8}$ and $n > 10$. Since the weight $k = \frac{n}{2} \equiv 1 \bmod{4}$, the space $\cS^{cm}_{n/2}(\Gamma_1(4))$ is generated by a single CM cusp form $C(q)$ constructed below.
	
	Let $\psi$ denote the algebraic Hecke character unramified away from 2 on $\bQ(i)$ with $\infty$-type $\frac{n}{2} - 1$. Viewing it adelically, $\psi$ acts on $\bC^\times$ by sending $z  \mapsto z^{1-k}$, and on primes $\fp \in \cO_K$, $\psi$ sends $\fp \mapsto \fp^{k-1}$, where $k = \frac{n}{2}$ as usual. The $L$-series attached to this character is 
	\begin{eqnarray*} L(s,\psi) &=& \left(1 + (1+i)^{k-1}2^{-s}\right)^{-1} \prod_{p \equiv 3 (4)} \frac{1}{1 - p^{k-1}p^{-2s}} \prod_{\stackrel{p \equiv 1 (4)}{p = \fp\overline{\fp}}} \frac{1}{\left(1 - \fp^{k-1}p^{-s}\right)\left(1 - \overline{\fp}^{k-1}p^{-s}\right)}\\
					  &=& \frac{1}{4} \sum_{m\geq 1}\left(\sum_{\stackrel{\fd \in \bZ[i]}{\Nm(\fd) = m}} \fd^{k-1}\right)\frac{1}{m^s}. \end{eqnarray*}

Applying an inverse Mellin transform gives the $q$-series
	\begin{eqnarray}\label{heckecusp} 
C(q) &=& \frac{1}{4}\sum_{m \geq 1}\left(\sum_{\stackrel{\fd \in \bZ[i]}{\Nm(\fd) = m}}\fd^{k-1}\right)q^m \\ 
	&=& q + (-4)^{k-1}q^2 + 2^{k-1}q^4 + ... \nonumber \end{eqnarray}
By Weil's converse theorem, this is a cusp form on $\Gamma_1(4)$ of weight $k$, and by construction, it has complex multiplication by the Dirichlet character $\chi_4$ attached to $\bQ(i)$ (see \cite{miyake}, \S 4.7, and \cite{milneMF}, Ch. 9). Thus $C(q) \in \cS^{cm}_k(\Gamma_1(4))$ and in fact, generates the space. 

Thus, $\theta_n$ is not elementary if the determinant of the matrix of coefficients for $q$, $q^2$, $q^3$, and $q^4$,
	$$\det(M'') = \begin{array}{| c c c c | cl}  
 	2n & 4{n \choose 2} & 8{n \choose 3} &16{n \choose 4} + 2n & & (\mbox{\small \em coefficients of  $\theta_n$})\\
	1 & 2^{n/2 - 1} & -1 + 3^{n/2 - 1}  & 4^{n/2 - 1} & & (\mbox{\small \em coefficients of  $E_1(q)$})\\
	1 &  1  & 1 - 3^{n/2 - 1} & 1 & & (\mbox{\small \em coefficients of  $E_2(q)$}) \\
	1 & (-4)^{\frac{n - 2}{8}} & 0 &  2^{\frac{n-2}{2}}  & & (\mbox{\small \em coefficients of $C(q)$})\end{array} $$ 
has non-zero determinant.

A straightforward calculation demonstrates the fact that $\det(M'')$ as a function on $n$ has negative growth on the order of $\cO(12^{n/2 - 1})$ for $n$ large enough. For $n = 10$, $\det(M'') =0$, but for $n = 18$, $\det(M'') = -439,038,812,160.$ It is therefore easy to check that $\det(M'')$ is nonzero for all $n > 10$ such that $n = 2 \bmod{8}$. This proves the theorem using the dimension arguments from Lemmas \ref{dim} and \ref{hecke}. \qed

\section{$\theta_n(q)$ is not elementary for $n > 10$, redux} \label{proof2}

Using the construction of a basis for the CM cuspidal subspace in Corollary \ref{CM basis}, we can show that the Hecke action of the cuspidal part of $\theta_n(q)$ does not coincide with the action on elements of $\cS_{\frac{n}{2}}^{cm}(\Gamma_1(4))$ at the prime $p = 3$ for $n > 10$. This allows for another proof of Theorem \ref{main}. We first describe Hecke operators which vanish on CM cusp forms. 

\begin{lemma} For any $f \in \cS_k^{cm}(\Gamma_1(4))$ with $k > 1$, the Hecke operator $T_p(f) = 0$ for any prime $p \equiv 3\bmod{4}$. \end{lemma}

\begin{proof} Recall that Corollary \ref{CM basis} implies that $\cS_k^{cm}(\Gamma_1(4))$ is generated over $\bC$ by forms $f_{\bQ(i),\psi}(z)$ defined by (\ref{CM form}), where $\psi$ is a Hecke character with $\infty$-type $t$ divisible the level $N = 4$. In order for any $f_{\bQ(i),\psi}(z)$ to have CM by the quadratic character $\chi_4$, the $p$th coefficient must vanish, i.e. $a_p = 0$ for all inert primes $p$, for all CM eigenforms. Thus, for any prime $p \equiv 3 \pmod{4}$, $T_p(f)$ must necessarily vanish for any $f \in \cS_k^{cm}(\Gamma_1(4))$. \end{proof}

\noindent \textbf{Remark.} Note that this result can be easily generalized to any level $N$: for any $f \in \cS_k^{cm}(\Gamma_1(N))$ with $k > 1$, $T_p(f) = 0$ for any prime which is inert in all imaginary quadratic fields $K$ with discriminant $-d$ where $d \mid N$. It is also possible that for a larger set of primes, $T_p(f)$ vanishes (e.g. for $\Gamma_1(144)$, $T_p(f) = 0$ for all primes $p \equiv -1 \bmod{12}$ or equivalently, all primes which are inert in only $\bQ(i)$ and $\bQ(\sqrt{-3}$), see \cite{lacunarity}, Lemma 1). Furthermore, the converse to this statement is also true: any cusp form $f \in \cS_k(\Gamma_1(4))$ such that $T_p(f) = 0$ for all $p \equiv 3 \bmod{4}$ is an element of $\cS_k^{cm}(\Gamma_1(4))$ (see \cite{chebotarev}, Corollary 15.2).

Thus, if we show that the part of $\theta_n$ which vanishes on all three cusps of $\Gamma_1(4)$ has a nonzero coefficient for $q^p$ for some $p \equiv 3 \bmod{4}$, then $\theta_n$ is not elementary. Since $\theta_n \in \cM_{n/2}(\Gamma_1(4)) = \cE_{n/2}(\Gamma_1(4)) \oplus \cS_{n/2}(\Gamma_1(4))$, we can decompose $\theta_n$ as the sum of an Eisenstein series and cusp form, i.e. write $\theta_n(z) = E_{n/2}(z) + f_{n/2}(z)$. 

We use a result derived from the Siegel-Weil formula describing the coefficients of $E_{n/2}(z)$ (see \cite{Siegel}, pp. 373-376). Note that because we are focused on Hecke operators of odd primes, we can focus on the odd coefficients of the Fourier expansion. In order to write down formulas for the odd coefficients, we must first understand certain constants. 

Let ${\bf e_k}$ denote the $k$th Euler number and ${\bf b_j}$ denote the $j$th Bernoulli number, which are defined by the following identities
	$$\frac{2}{e^t + e^{-t}} = \sum_{k=0}^\infty {\bf e_k} \frac{t^k}{k!}, \quad \quad \frac{t}{e^t - 1} = \sum_{j=0}^\infty {\bf b_j} \frac{t^j}{j!}.$$
The magnitudes of these numbers are related to values of $L$-series associated to the primitive Dirichlet characters of conductor 1 and 4:
	$${\bf b_j} = \frac{2 \cdot j!}{(2\pi i)^j}\cdot \zeta(j) \quad \mbox{if $j > 0$ even,} \quad \mbox{and} \quad {\bf e_k} = \frac{2^{2k+3} \cdot k!}{(2\pi i)^{k+1}} \cdot L(k+1,\chi_4) \quad \mbox{if $k > 0$ even}.$$

\begin{theorem}[Shimura \cite{squaresshimura} and \cite{3shimura}] If $E_{n/2}(z) = 1 + \sum_{m=1}^\infty c_m q^m \in \cE_{n/2}(\Gamma_1(4))$ denotes the $q$-expansion of the Eisenstein part of $\theta_{n/2}(q)$, then for odd $m$,
$$c_m = \begin{cases} 
		\displaystyle\frac{4}{|{\bf e_{n/2 - 1}}|}\cdot\left(\chi_4(m)\cdot 2^{n/2 - 1} + \chi_4\left(\frac{n}{2}\right)\right) \cdot \sum_{d \mid m} \chi_4(d) d^{n/2 -1} & \mbox{ if $n \equiv 2 \pmod{4}$ and $n > 2$} \\
		\displaystyle \frac{n}{(2^{n/2} - 1)|{\bf b_{n/2}}|} \cdot \sum_{d \mid m} d^{n/2-1} & \mbox{ if $n \equiv 0 \pmod{4}$} \end{cases}$$ \end{theorem}

This result comes from using local computations on the genus and the mass of the standard lattice $\bZ^n$ to understand the ``dominant'' growth factor $c_m$ of $r_n(m)$. To demonstrate that $r_n(m)$ is generally governed by the growth of $c_m$, one shows that $c_m - r_n(m)$ are coefficients of a cusp form, thus $1 + \sum_{m=1}^\infty c_mq^m$ is the Eisenstein part of the $\theta_n(q)$.

\begin{lemma} If $f_{n/2}(z)$ denotes the cuspidal part of $\theta_n(z)$, then the 3rd coefficient $a_3(f_{n/2}) \neq 0$ if $n > 10$. \end{lemma}

\begin{proof}  One can easily compute that $r_n(3) = 8{n\choose 3} = \frac{4n(n-1)(n-2)}{3}$. We can then calculate the third coefficient of the $q$-series expansion of $f_{n/2}(z) = \theta_n(z) - E_{n/2}(z)$ using this equality:
	$$f_{n/2}(z) = \sum_{m=1}^\infty a_mq^m = \sum_{m=0}^\infty r_n(m)q^m - \sum_{m=0}^\infty c_mq^m \in \cS_{n/2}(\Gamma_1(4)).$$
Thus, taking the difference of $c_3$ and $r_n(3)$ gives the coefficients of $f_{n/2}(z)$:
	$$a_3 = \begin{cases}
	\displaystyle \frac{4n(n-1)(n-2)}{3} + \left(\frac{4}{|{\bf e_{n/2 - 1}}|}\right) \cdot \left(2^{n/2 - 1} - \chi_4\left(\frac{n}{2}\right)\right) \cdot \left(1 - 3^{n/2 - 1}\right)  & \mbox{ if $n \equiv 2 \pmod 4$} \\
	\displaystyle \frac{4n(n-1)(n-2)}{3} + \left(\frac{n}{|{\bf b_{n/2}}|}\right) \cdot \left(\frac{1 + 3^{n/2 - 1}}{1 - 2^{n/2}}\right) & \mbox{ if $n \equiv 0 \pmod{4}$} \end{cases}$$
Note that the values of $|{\bf b_j}|$ and $|{\bf e_k}|$ for even $j$ and $k$ have bounds given below (see \cite{leeming}), 
	$$4\sqrt{\frac{\pi j}{2}} \left(\frac{j}{2\pi e}\right)^j < |{\bf b_j}|, \quad \text{ and } \quad 8\sqrt{\frac{k}{2\pi}} \left(\frac{2k}{\pi e}\right)^{k} < |{\bf e_k}|.$$
Thus, we can bound $a_3$ as a function on $n$ by
\begin{eqnarray*}
	a_3(n) &>& \frac{4n(n-1)(n-2)}{3} + \left(\frac{1}{2}\sqrt{\frac{2\pi}{n/2 - 1}}\left(\frac{\pi e}{n - 2}\right)^{n/2 -1}\right) \cdot \left(2^{n/2 - 1} - 1\right) \cdot \left(1 - 3^{n/2 - 1}\right), \\
	a_3(n) &>& \frac{4n(n-1)(n-2)}{3} + \left(\frac{n}{4}\sqrt{\frac{4}{\pi n}} \left(\frac{4 \pi e}{n}\right)^{n/2}\right) \cdot  \left(\frac{1 + 3^{n/2 - 1}}{1 - 2^{n/2}}\right). 
\end{eqnarray*}

It is easy to check that for $n$ large enough, the non-polynomial part of the lower bounds has negative growth on the order of $\cO(n^{-n/2})$ which is dominated by the positive growth of the polynomial portion. Some asymptotic calculations further demonstrate that these functions on $n$ are monotonically increasing for $n > 18$ and $20$, so in particular $a_3$ is nonzero for $n > 10$ as seen in the following table.) 

\begin{center}
\begin{tabular}{|c|c|c|c|}\hline ${\bf n}$ & ${\bf c_3}$ & ${\bf r_n(3)}$ & ${\bf a_3 = r_n(3) - c_3}$ \\\hline \hline 4 & 32 & 32 & 0 \\\hline 6 & 160 & 160 & 0 \\\hline 8 & 448 & 448 & 0 \\\hline 10 & 960 & 960 & 0 \\\hline 12 & 1952 & 1760 & -192 \\\hline 14 & 189280/61 & 2912 & -11648/61 \\\hline 16 & 70016/17 & 4480 & 6144/17 \\\hline 18 & 1338240/277 & 6528 & 470016/277 \\\hline 20 & 157472/31 & 9128 & 125248/31 \\\hline \end{tabular} \end{center}
Thus, the cuspidal part of $\theta_n$ has a nonzero coefficient for $q^3$ in the Fourier expansion. The first $q$-series coefficient of $T_3(s_{n/2})$ is $a_3$ which is nonzero for $n > 10$. \end{proof}

Combining the two lemmas, we conclude that the cuspidal part of $\theta_n$ does not lie in the CM subspace of $\cS_{n/2}(\Gamma_1(4))$ for $n > 10$, hence $\theta_n$ cannot be elementary for these cases. Theorem \ref{main} again follows from showing that $\theta_n$ is elementary for $n = 2, 4, 6, 8, 10$ from Section \ref{exceptional}.

\section{Motivation for definition of elementary modular forms} 

While the notion of an elementary modular form coincides with the property of having lacunary cuspidal part as well as being attached to potentially abelian $\ell$-adic representations of $\Gal(\overline{\bQ}/\bQ)$, the motivation behind the definition of ``elementary'' is computational: an elementary modular form should have Fourier coefficients which have formulas that can be computed in polynomial time in $\log(m)$ (for the coefficient of $q^m$). Since the $m$th coefficient of an Eisenstein series is usually given as a sum over the positive divisors of $m$, one can weaken this by assuming factorization of $m$ is given when calculating the coefficient of $q^m$ in log-polynomial time. For $\theta_n$, this occurs for even $n < 12$, demonstrated by the formulas given in these exceptional cases, and these allow for efficiently computing $r_{n}(m)$, even for $\theta_{10}$ which is not purely Eisenstein. 
To demonstrate that modular forms on $\Gamma_1(4)$ that involve a non-CM cusp form are not ``elementary,'' we consider the case of $n = 12$.

\subsection{$n$ = 12} 

Note that the subspace of $\cM_6(\Gamma_1(4))$ consisting of Eisenstein series has dimension 3, and the subspace of cusp forms has dimension 1 by Lemma \ref{dim}. Furthermore, by Lemma \ref{hecke}, $\cS^{cm}_6(\Gamma_1(4))$ is trivial, and we can in fact use a well-known cusp form for the basis element of $\cS_6(\Gamma_1(4))$: 
	$$\sqrt{\Delta}(2z) = \ \eta^{12}(2z) \ := \  q\prod_{n=1}^\infty (1 - q^{2n})^{12} = q \left[1  - 12q^2 + 54q^4 - 88q^6 - 99q^8 + ...\right]$$ 
The Dedekind eta function $\eta(z)$ is a cusp form on $\Gamma_1(2)$ of weight $1/2$, hence $\sqrt{\Delta}(2z) \in \cM_6(\Gamma_1(4))$. It is a cusp form without complex multiplication. 

From above and (\ref{evenweight}), take $\{E(q), E(q^2), E(q^4), \sqrt{\Delta}(q^2)\}$ with $k = 6$ as a basis for $\cM_6(\Gamma_1(4))$. We then calculate the constants $a,b,c,d \in \bC$ such that $$\theta_{12}(q) = a \cdot E(q) + b \cdot E(q^2) + c \cdot E(q^4) + d \cdot \sqrt{\Delta}(q^2).$$ Note that by Theorem \ref{main}, $d$ must be nonzero. The first few coefficients of the $q$-series in the basis are as follows: 
\begin{eqnarray*}
E(q) &=& -\frac{\bf b_k}{2k} + q + 33\cdot q^2 + 244 \cdot q^3 + 1057 \cdot q^4 + ... \\
E(q^2) &=& -\frac{\bf b_k}{2k} + q^2 + 33 \cdot q^4 + ... \\
E(q^4) &=& -\frac{\bf b_k}{2k} + q^4 + ... \\
\sqrt{\Delta}(q^2) &=& q - 12\cdot q^3 + ... 
\end{eqnarray*}
 Since we know that the first few coefficients $r_{12}(m)$ are $$\theta_{12}(q) = 1 + 24\cdot q + 264 \cdot q^2 +  1760 \cdot q^3 + 7944 \cdot q^4 + ...,$$ thus $a = 8$, $b = 0$, and $c = -512$, $d = 16$.
 
Here, we can conclude that it is necessary to calculate the coefficients of $\eta^{12}$ in order to calculate $\theta_{12}(n)$. Furthermore, the converse is true. 
 
From Serre's point of view, $\eta^{12}$ is not lacunary, i.e. it is not a CM form, thus the coefficients cannot be given via characters of a CM field (see \cite{chebotarev}, \S 15). Furthermore, it has a positive density of non-zero coefficients in its $q$-series that contribute to $\theta_{12}$ and are hard to compute (see \cite{lacunarity}). The formulas for $n = 2, 4, 6, 8,$ and $10$ illustrate that $r_n(m)$ can be calculated efficiently if the prime factorization of $m$ is known (thus for a prime $p$, $r_n(p)$ can be computed in $\log(p)$-time). In the case of $n = 12$, this is not enough information as there is no analogous description of the coefficients of $\eta^{12}(2z)$ in terms of divisors. One must know how to compute the Ramanujan $\tau$ function in order to calculate $r_{12}(m)$.
 
 \subsection{Making non-elementary $\theta_n(q)$ computable}
 
Recently, Bas Edixhoven, Jean-Marc Couveignes, and Robin de Jong have proven that if $f = \sum a_n q^n$ is a modular form on $\SL_2(\bZ)$, then even when there is not an elementary formula for the Fourier coefficients, $a_p$ for $p$ prime can be computed in time polynomial in the weight $k$ and $\log(p)$ assuming GRH (see \cite{couveignes}). This implies that the prime Fourier coefficients $\tau(p)$ of $\Delta$ can be calculated in polynomial time with respect to $\log(p)$. Furthermore, Peter Bruin gives a description of a probabilistic algorithm in time polynomial to $k$ and $\log(p)$ which under the assumption of GRH, computes the Fourier coefficient of $q^p$ of eigenforms of level $2N'$ where $N'$ is squarefree (see \cite{bruin}). This includes calculating $r_{n}(m)$ for all even $n$ as discussed here, whether or not $\theta_n$ has an elementary formula. In particular, since $\theta_n$ is not elementary for all even $n > 10$, and therefore does not have a nice formula, such an algorithm is clearly necessary. Note that for elementary $\theta_n$, the formulas given here compute coefficients faster than Bruin's algorithm. Currently, Bruin's algorithm is largely theoretical and future work will be done to make it more practical. Other than understanding classical arithmetic functions, these algorithmic results are useful in computing eigenvalues of Hecke operators or equivalently, coefficients of eigenforms (see \cite{edixhoven}, \cite{couveignes}, and \cite{bruin}).

\nocite{bosman} \nocite{rankin} \nocite{ddt} \nocite{milnes} \nocite{milneMF} \nocite{rankin} \nocite{chebotarev} \nocite{pinkshimura} \nocite{squaresshimura} \nocite{jacobianshimura} \nocite{realshimura} \nocite{serreconjecture} \nocite{smith} \nocite{hurwitz} \nocite{ogg}
\bibliographystyle{plain}
\bibliography{Mthesisbib}{}
								  
\end{document}